\documentclass[12pt]{amsart}
\usepackage{amssymb}
\usepackage{mathptmx}
\usepackage[latin1]{inputenc}

\let\oldlabel=\label
\def\prellabel{\marginparsep=1em\marginparwidth=44pt
\def\label##1{\oldlabel{##1}\ifmmode\else\ifinner\else
\marginpar{{\footnotesize\ \\ \tt
##1}}\fi\fi}}

\usepackage[active]{srcltx} 
\usepackage[latin1]{inputenc}
\usepackage[all]{xy}

\newcommand{\mm}{\mathfrak m}
\newcommand{\N}{\mathbb{N}}

\newcommand{\Q}{\mathbb{Q}}

\newcommand{\ff}{{\bf f}}
\newcommand{\bg}{{\bf g}}

\newcommand{\ub}{{\bf u}}
\newcommand{\jb}{{\bf j}}

\newcommand{\vb}{{\bf v}}

\DeclareMathOperator{\depth}{depth}

 \DeclareMathOperator{\ini}{in}

\DeclareMathOperator{\reg}{reg}

\DeclareMathOperator{\Tor}{Tor}

\DeclareMathOperator{\dirsum}{\oplus}

\DeclareMathOperator{\pnt}{\raise 0.5mm \hbox{\large\bf.}}

\newtheorem{thm}{\bf Theorem}[section]
\newtheorem{lem}[thm]{\bf Lemma}
\newtheorem{cor}[thm]{\bf Corollary}
\newtheorem{prop}[thm]{\bf Proposition}

\theoremstyle{definition}

\newtheorem{rem}[thm]{\bf Remark}
\newtheorem{ex}[thm]{\bf Example}

\begin{document}

\title{Regularity bounds for  Koszul cycles}
\author{Aldo Conca} 
\address{Dipartimento di Matematica, Universit\'{a} di Genova, Via Dodecaneso 35, 16146 Genova,
 Italy} \email{conca@dima.unige.it}

\author{Satoshi Murai}
\address{
Satoshi Murai,
Department of Mathematical Science,
Faculty of Science,
Yamaguchi University,
1677-1 Yoshida, Yamaguchi 753-8512, Japan.
}

\email{murai@yamaguchi-u.ac.jp}



\thanks{Research of the second author was partially supported by KAKENHI 22740018}

\dedicatory{ }

\begin{abstract}
We study the Castelnuovo-Mumford regularity of the module of Koszul cycles $Z_t(I,M)$ of a homogeneous ideal $I$ in a polynomial ring $S$    with respect to a graded module $M$ in the homological position $t\in \N$. Under mild assumptions on the base field we prove in Theorem \ref{th1}  that $\reg Z_t(I,S)$ is a subadditive function of $t$ when $\dim S/I=0$. For Borel-fixed ideals $I,J$ we prove in Theorem \ref{thm2} that $\reg Z_t(I,S/J)\leq t(1+\reg I)+\reg S/J$,  a result already announced in \cite{BCR2} by Bruns, Conca and R\"omer. \end{abstract}

 \maketitle

%
%
%

\section*{Introduction} Let $S$ be a polynomial ring over a field $K$, say of characteristic $0$ for simplicity. Let $I\subset S$ be a homogeneous ideal of $S$ and $M$ a finitely generated graded module. Denote by $\reg M$ the Castelnuovo-Mumford regularity of $M$.    Denote by $K(I,M)$ the Koszul complex associated to a minimal system of generators of $I$ with coefficients in $M$. Let $Z_t(I,M)$ be the $S$-module of  cycles of homological position $t$ of $K(I,M)$. If there is no danger of confusion, we simply denote by $Z_t$ the module $Z_t(I,S)$. By construction $Z_1$ is the first syzygy module of $I$ and so by definition we have
$$\reg Z_1=1+\reg I$$
unless $I$ is principal (in that case $Z_1=0$ and it has regularity $-\infty$  by convention). 
 
Our study of regularity bounds for the Koszul cycles and homology has its motivations and origin in the work of  Green \cite{G} who   proved   (among other things) a regularity bound for the Koszul homology of the powers of the maximal ideal in a polynomial ring. 
Green's result gives a  bound for the degrees of the syzygies of the Veronese varieties. In  \cite{BCR1} and  \cite{BCR2}  better regularity bounds for Koszul cycles and homology have been proved and that led to an improvement of our understanding of the syzygies of Veronese varieties.  In particular, generalizing results of \cite{BCR1},  in  \cite[Prop.3.2]{BCR2} it is shown  that 
\begin{equation}
\label{eq:1} \reg Z_t(I,M)\leq t(1+\reg I)+\reg(M) 
\end{equation}
holds for every $t$  when $\dim M/IM\leq 1$ and examples are given showing that Eq.(\ref{eq:1}) does not hold in general. 
It is also asked in \cite{BCR2} whether the inequality
\begin{equation}
\label{eq:2} \reg Z_t \leq t(\reg I+1) 
\end{equation}
does hold in general. In this paper we give examples showing that  Eq.(\ref{eq:2}) does not hold in general but we  show that in two special cases variants of  Eq.(\ref{eq:1}) and   Eq.(\ref{eq:2}) do hold. In details,  we show that if $\dim S/I=0$ then 
\begin{equation}
\label{eq:3} \reg Z_{s+t} \leq  \reg Z_{s}+  \reg Z_{t}
\end{equation}
holds for all $s$ and $t$.   And we also prove that 
\begin{equation}
\label{eq:4} \reg Z_t(I,S/J)\leq t(\reg I+1)+\reg(S/J) 
\end{equation}
holds whenever $I$ and $J$ are Borel-fixed ideals. This result was already announced in \cite[Thm.3.8]{BCR2}.

\section{Generalities} 
\label{} 
In this section we collect notation and general facts about Koszul complexes. As a general reference for facts concerning Koszul complex and homology the reader can consult for instance Bruns and Herzog \cite[Chap.1]{BH}.   

Let $S=K[x_1,\dots,x_n]$ be a polynomial ring over a field $K$.  The maximal homogeneous ideal $(x_1,\dots,x_n)$ of $S$ is denoted by $\mm_S$ or just by $\mm$. Let $I\subset S$ be an ideal minimally generated by homogeneous polynomials  $f_1,\dots,f_m$.  Denote by $K(I,S)$ the Koszul complex associated to the $S$-linear map $\phi: F=\oplus S(-\deg f_i) \to S$ defined by $\phi(e_i)=f_i$. Given a graded $S$-module $M$ we set $K(I,M)=K(I,S)\otimes M$. We consider both $K(I,S)$ and $K(I,M)$  graded complexes with maps of degrees $0$. 
 We have decompositions  $K(I,S)=\bigoplus_{t=0}^m K_t(I,S)=\bigwedge^\bullet F$ and
$K(I,M)=\bigoplus_{t=0}^mK_t(I,M)=\bigwedge^\bullet F\otimes M$. 
The complex  $K(I,M)$ can be seen as a graded module over the exterior algebra $K(I,S)$. For $a\in K(I,S)$ and $b\in
K(I,M)$ the multiplication will be denoted by $a.b$. The
differential of $K(I,S)$ and $K(I,S)$ will be denoted simply
by $\phi$ and it satisfies
$$\phi(a.b)= \phi(a).b + (-1)^s a.\phi(b)$$
for all $a\in K_s(I,S)$ and $b\in K(I,M)$. We let
$Z_t(I,M)$, $B_t(I,M)$, $H_t(I,M)$  denote the cycles, the
boundaries and the homology in homological position $t$ and set
$Z(I,M)=\dirsum Z_t(I,M)$ and so on. One knows that
$Z(I,S)$ is a (graded-commutative) $S$-subalgebra of $K(I,S)$ and that $B(\phi, R)$ is
a homogeneous  ideal of $Z(I,S)$ so that the homology $H(I,S)$ is itself
a (graded-commutative) $S$-algebra. More generally,  $Z(I,M)$ is a $Z(I,S)$-module. We
will denote by $Z_s(I,S)Z_t(I,M)$ the image of the
multiplication map $Z_s(I,S)\otimes Z_t(I,M)\to
Z_{s+t}(I,M)$. Similarly, $Z_1(I,S)^t$ will denote the image
of the map $\bigwedge^t Z_1(I,S)\to Z_{t}(I,S)$.

By construction, Koszul cycles, boundaries and homology have an induced graded structure. 
An index on the left of a graded module always denotes the selection of the homogeneous
component of that degree.

Denote by $\{e_1,\dots,e_m\}$ the canonical basis of the free $S$-module  $F=\oplus S(-\deg f_i)$, so that $\deg e_i=\deg f_i$. Given
$\ub=\{u_1,\dots,u_s\}\subset [m]$ with $u_1<u_2<\dots<u_s$ we write
$e_\ub$ for the corresponding basis element $e_{u_1}\wedge \cdots
\wedge e_{u_s}$ of $\bigwedge^s F$.  Alternatively we use the symbol $[f_{u_1},\dots,f_{u_s}]$ to denote $e_\ub$ which is a homogeneous element of degree $\sum_i  \deg f_{u_i}$.

Any element $g\in \bigwedge^s F\otimes M$ can be written uniquely as
$g=\sum e_\ub \otimes m_\ub$ with $m_\ub\in M$ where the sum is over the
subsets of cardinality $s$ of $[m]$. If $m_\ub= 0$ then we will say
that $e_\ub$ does not appear in $g$. For every $g\in K_{s+t}(I, M)$
and for every $\ub\subset [m]$ with $s=\# \ub$ we have a unique
decomposition
\begin{equation}
\label{deco} g=a_\ub(g)+ e_\ub.b_\ub(g)
\end{equation}
with $a_\ub(g)\in K_{s+t}(I, M)$ and $b_\ub(g)\in K_{t}(I, M)$  provided we require that   $e_\jb$ does not appear in $a_\ub(g)$ whenever $\jb\supset \ub$
and that $e_\vb$ does not appear in $b_\ub(g)$ whenever $\vb\cap \ub\neq \emptyset$.
With the notation above in \cite[Lemma 2.2 and 2.4]{BCR2} it is proved that: 
\begin{lem}
\label{le1}
\item[(1)] If  $g \in Z_{s+t}(I, M)$ then $b_\ub(g)\in Z_{t}(I, M)$
for every $\ub$ with $s=\# \ub$.
\item[(2)]  The assignment  $$\beta_t(g)= \sum_\ub e_\ub\otimes b_\ub(g)$$
where the sum is over the $\ub\subset [m]$ with $\#\ub=s$  gives a homomorphism 
$$\beta_t:  Z_{s+t}(I, M)\to Z_{s}(I, Z_t(I,M))$$  
of $S$-modules. 
\item[(3)]  The assignment  $$\alpha_t(\sum_i a_i\otimes g_i)= \sum_i  a_i.g_i$$
 gives a  homomorphism 
$$\alpha_t:   Z_{s}(I, Z_t(I,M)) \to Z_{s+t}(I, M)$$ 
of $S$-modules. 
\item[(4)] The composition $\alpha_t \circ \beta_t$ is the multiplication by the $\binom{t+s}{s}$. Hence $Z_{s+t}(I, M)$ is a direct summand of $Z_{s}(I, Z_t(I,M))$ as an $S$-module provided  $\binom{t+s}{s}$ is invertible in $K$.
\end{lem}
An easy but interesting fact: 
\begin{lem}
\label{ZZtop}
Let $s,t\in \N$. With the notation introduced above one has 
$$Z_s(I,Z_t(I,S))=Z_t(I,Z_s(I,S))$$ 
where both sets are interpreted as subsets of $\wedge^s F\otimes \wedge^t F$.
\end{lem} 
 \begin{proof} Let $$g=\sum_{\alpha,\beta}  a_{\alpha,\beta} e_{\alpha}\otimes e_{\beta} \in \wedge^s F\otimes \wedge^t F$$ 
where $a_{\alpha,\beta}\in S$ and  $\alpha$ varies in the set of subsets of cardinality $s$ on $[m]$ and $\beta$ varies in the set of subsets of cardinality $t$ on $[m]$. One has $g\in  Z_s(I, \wedge^t F ))$  if and only if 
$$\sum_{\alpha,\beta}  a_{\alpha,\beta} \phi(e_{\alpha})\otimes e_{\beta}=0$$
that is, 
$$\sum_{\alpha}  a_{\alpha,\beta} \phi(e_{\alpha})=0 \mbox{ for every } \beta$$
that is
\begin{equation}
\label{simme1}
\sum_{\alpha}  a_{\alpha,\beta} e_{\alpha}\in Z_s(I,R)\mbox{ for every } \beta
\end{equation} 
Furthermore $g\in  \wedge^s F\otimes Z_t(I,S)$ if and only if 
\begin{equation} 
\label{simme2}
\sum_{\beta}  a_{\alpha,\beta}   e_{\beta}\in  Z_t(I,S) \mbox{ for every } \alpha
\end{equation}
It follows that $g\in Z_s(I,Z_t(I,S)$ if and only if Eq.(\ref{simme1}) and Eq.(\ref{simme2}) hold. 
Symmetrically, $g\in  Z_t(I, \wedge^s F ))$ if and only if Eq.(\ref{simme2}) holds and $g\in   Z_t(I,S)\otimes  \wedge^t F$ if and only if Eq.(\ref{simme1})  holds. 
\end{proof} 

\begin{rem} As the proof shows the statement of  Lemma \ref{ZZtop} holds for every Noetherian ring. \end{rem} 

The following result allows  us, when studying  Eq.(\ref{eq:1}), to assume that the ideals we deal with have a linear resolution.

\begin{prop} 
\label{linres}
Let $I$ be a homogeneous ideal and $M$ a graded $S$-module. Let $d=\reg I$ and set $J=(I_d)$ (so that $\reg J=d$). 
Then Eq.(\ref{eq:1}) holds for $I$ and $M$ and every $i$  if it holds for  $J$ and $M$ and every $i$. 
\end{prop}

In order to prove Proposition \ref{linres} we need some auxiliary results. To this end we introduce a piece of notation. 
Given a sequence of homogeneous polynomials  $\ff=f_1,\dots,f_m$ we will denote by $K(\ff,M)$ the Koszul complex associated to the sequence $\ff$ with coefficients in $M$. And we denote by $Z(\ff,M)$ the cycles and so on. 
 Note that here we do not assume that the $f_i$ are a minimal system of generators of the ideal they generate. 
 
 We have: 
 
 \begin{lem} 
\label{lem1}
Let  $I=(\ff)$ and  $g_1,\dots,g_v\in I$.  Set $\bg=g_1,\dots,g_v$. Then $\reg Z_i(\ff,M)\leq \reg Z_i(\ff, \bg,M)$. 
 \end{lem}
 \begin{proof} By induction on $v$, it is enough to prove the statement for $v=1$. The assertion is obvious since $Z_i(\ff, \bg,M)\simeq Z_i(\ff,M)\oplus Z_{i-1}(\ff,M)(-\deg g_1)$.
 \end{proof}

  \begin{lem} 
\label{lem2}
Let  $I=(\ff)$ and let $a_1,\dots,a_v\in \N$.   Let  $g_i\in I:\mm_S^{a_i}$.  Set $\bg=g_1,\dots,g_v$. 
Then 
$$\reg Z_i(\ff,\bg,M)\leq  \max\{\reg  Z_{i-\# D}(\ff,M)+\sum_{j\in D} (a_j+\deg g_j) \  :\  D\subseteq \{1,\dots,v\} \mbox{ and } \#D\leq i\}.$$
 \end{lem}
 
 \begin{proof} By induction on $v$, it is enough to prove the statement for $v=1$. Set $a=a_1$ and $g=g_1$. 
 Let $\alpha:Z_i(\ff,\bg,M)\to Z_{i-1}(\ff,M)(-\deg g)$ be the map defined by $\alpha(h)=h_{1}$ where $h=h_0+e_{m+1}h_1$ with  $h_0\in K_i(\ff,M)$. Set $p=a+\deg g+\reg Z_{i-1}(\ff,M)$.  We claim that $\alpha$ is surjective in degrees $\geq p$. Let $h_1\in Z_{i-1}(\ff,M)(-\deg g)$ with $\deg h_1\geq p$. In other words, $h_1$ has degree $\geq a+\reg Z_{i-1}(\ff,M)$ as an element of $Z_{i-1}(\ff,M)$. That is, $h_1\in \mm^a Z_{i-1}(\ff,M)$. Hence 
 $$gh_1\in g\mm^a Z_{i-1}(\ff,M)\subset I Z_{i-1}(\ff,M)\subset B_{i-1}(\ff,M).$$
 Therefore  it exists $w\in K_i(\ff,M)$ such that $\phi(w)=gh_1$. This implies that $-w+e_{m+1}h_1\in Z_i(\ff,\bg,M)$ and hence $\alpha(-w+e_{m+1}h_1)=h_1$. 
 It follows that the complex
 $$0\to Z_i(\ff,M)\to Z_i(\ff,\bg,M)\to Z_{i-1}(\ff,M)(-\deg g)\to 0$$
 is exact in degrees $\geq p$. We deduce that: 
 $$\reg  Z_i(\ff,\bg,M)\leq \max\{ p, \reg Z_i(\ff,M), \reg  Z_{i-1}(\ff,M)+\deg g\}=\max\{ p, \reg Z_i(\ff,M)\}.$$ 
  \end{proof}
    
 We are now ready to prove Proposition \ref{linres}: 
  
  \begin{proof}[Proof of \ref{linres}] Let $g_1,\dots,g_v$ be the minimal generators of $I$ of degree $<d$ and let $f_1,\dots,f_m$ be the generators of $I_d$. Set $\ff=f_1,\dots,f_m$, $\bg=g_1,\dots,g_v$. Since the sequence $\ff,\bg$ contains a minimal system of generators of $I$   by  Lemma \ref{lem1} we have $\reg Z_i(I,M)\leq \reg Z_i(\ff,\bg,M)$.  Since $g_i\mm^{d-\deg g_i}\subset (I_d)=J$ we may use  Lemma \ref{lem2} and get 
  
$$\reg Z_i(I,M)\leq \reg Z_i(\ff,\bg,M)\leq  \max\{ jd+\reg  Z_{i-j}(J,M) : j\leq i\}.$$
But, by assumption,  
$$\reg  Z_{i-j}(J,M)\leq (i-j)(\reg J+1)+\reg M=(i-j)(d+1)+\reg M.$$
It follows that 
$$\reg Z_i(I,M)\leq  \max\{ jd+(i-j)(d+1)+\reg M : j\leq i\}=i(d+1)+\reg M.$$
  \end{proof} 
 
 \section{Examples} 
 We present in this section some examples of ideals which do not satisfy the inequality Eq.(\ref{eq:2}). 
 They are all defined by cubics, with a linear resolution and the failure of Eq.(\ref{eq:2}) comes from the fact that some boundaries are  minimal generators of the module of $2$-nd cycles.

 \begin{ex}
 \label{e1}
  Let $I$ be the ideal of the minors of size $3$ of a $3\times 5$ matrix $X=(x_{ij})$ of variables and $S=K[x_{ij}]$ so that  $\reg I=3$.   The module $Z_2=Z_2(I,S)$ (computed with CoCoA \cite{Co})  has $105$ generators of degree $8$ and $50$ generators of degree $9$. The generators of degree $9$ are indeed boundaries (i.e. the homology in generated in degree $8$).   
 The Betti table of $Z_2$ is
 $$\begin{array}{ccccccccccc}
      &  0  &  1  &  2  &  3  &  4  &   5  &  6 \\
\hline 
\hline
8:  & 105 &   90 &  21  &  - &   -  &  -  &  - \\
9:  &  50  & 225 & 420 & 420 &  240  & 75 &  10
 \end{array}
 $$
 So we have that $\reg Z_2= 9>2(\reg I+1)=8$.
\end{ex}

 \begin{ex} Let $J$ be the ideal of the leading terms of the ideal $I$ of  Example \ref{e1} with respect to a diagonal term order, i.e. $J=( x_{1i_1}x_{2i_2}x_{3i_3} : 1\leq i_1<i_2<i_3\leq 5)$.  Then $\reg J=3$ and $Z_2=Z_2(J,S)$ has minimal  generators of degree $9$  that boundaries (i.e. the homology in generated in degree $\leq 8$).   
 The Betti table of $Z_2$ is
 $$\begin{array}{ccccccccccc}
      &  0  &  1  &  2  &  3  &  4    \\
\hline 
\hline
7:  &      3&  - &   -  &  -  &  -  \\
8:  & 102 &  101 &  42  &  12 &   2   \\
9:  &    6  &    21 & 27 &     15 &  3   
 \end{array}
 $$
 So we have that $\reg Z_2= 9>2(\reg J+1)=8$.
\end{ex} 
 \begin{ex} Consider the ideals $$J_1=(x_1x_2x_3, x_1x_4x_6, x_3x_4x_5, x_4x_5x_6, x_1x_2x_6, x_1x_3x_4, x_2x_3x_5)$$ 
 and $$J_2=(x_2x_3x_6, x_1x_2x_6, x_1x_3x_5, x_1x_4x_5, x_3x_5x_6, x_1x_2x_5, x_3x_4x_6).$$ 
They have both a linear resolution.  The Betti tables of the corresponding $Z_2(J_i,S)$ are, respectively, 
 $$\begin{array}{ccccccccccc}
      &  0  &  1  &  2  &  3      \\
\hline 
\hline       
 8:  &   36  &  27  &  6  &  - \\
 9:   &  1    &  3    & 3   &  1
\end{array} 
$$ 
and 
$$
\begin{array}{ccccccccccc}
      &  0  &  1  &  2  &  3      \\
\hline 
\hline 
 7:  &   2  &  -  &  -  &  -\\
 8:  &  30 &  21&  4 & -\\
 9:  &   1  &  3  &  3 &   1
 \end{array} 
$$
so that $\reg Z_2(J_i,S)=9>8=2(\reg J_i+1)$. 
In both cases the generator of degree $9$ of $Z_2$ is a boundary, corresponding to the triplet 
$\{x_2x_3x_5, x_1x_2x_6, x_4x_5x_6\}$ in the first case and 
$\{x_3x_4x_6, x_1x_4x_5, x_1x_2x_6\}$ in the second. 
\end{ex}

 \section{The $0$-dimensional case} 
 
 The goal of this section is to prove the following 
 
 \begin{thm} 
 \label{th1}
 Assume $\dim S/I=0$ and the characteristic of $K$ is either $0$ or $>s+t$. Then  
$$\reg Z_{s+t} \leq \reg Z_{s}+\reg Z_{t}$$
holds. 
 \end{thm}

Indeed we prove

\begin{prop}
\label{super} Let $S$ be a polynomial ring of characteristic  $0$ or $>s+t$. Assume that $M$ is graded, finitely generated  with $\depth M>0$  and  $\dim S/I=0$. 
Then $$\reg Z_{s+t}(I,M)\leq  \reg Z_t+\reg Z_s(I,M).$$
\end{prop} 

\begin{proof}  First note that since  by Lemma \ref{le1}  $Z_{s+t}(I,M)$ is a direct summand of $Z_t(I, Z_s(I,M))$ we have 
$$\reg  Z_{s+t}(I,M)\leq \reg Z_t(I, Z_s(I,M)).$$
The canonical map $f:Z_t\otimes Z_s(I,M) \to Z_t(I, Z_s(I,M)$  becomes an isomorphism when localized at a relevant homogeneous prime because $\dim S/I=0$. Hence $f$ has $0$-dimensional kernel and cokernel. Since $Z_t(I, Z_s(I,M)$ is a submodule of a direct sum of copies of $M$ it has positive depth. It follows that $\reg Z_t(I, Z_s(I,M)\leq \reg Z_t\otimes Z_s(I,M)$. We observe that $\Tor_1^S(Z_t,N)$  has Krull dimension $0$ for every $S$-module $N$ because $Z_t$ is free when localized at a relevant homogeneous prime. 
 So we may apply  \cite[Cor.3.4]{C} or \cite[Cor.3.1]{EHU} and get $\reg Z_t\otimes Z_s(I,M)\leq   \reg  Z_t+\reg Z_s(I,M)$  
and this concludes the proof.
\end{proof} 

Now Theorem \ref{th1} is a special case ($M=S$) of  Proposition \ref{super}. 
We may deduce from  Theorem \ref{th1}  the following corollary concerning  the regularity of Koszul homology. 

 \begin{cor}
\label{subadkh} Let $S$ be a polynomial ring of characteristic  $0$ or $>s+t$. Assume $I$ is a homogeneous ideal with $\dim S/I=0$. 
Set $h_i=\reg H_i(I,S)$. We have: 
Then $$h_{s+t}  \leq s+t+1+\reg I+ \max\{ h_j-j : j<s\}+\max\{ h_j-j : j<r\}.$$
\end{cor} 

\begin{proof}  Set $z_i=\reg Z_i(I,S)$ and $b_i=\reg B_i(I,S)$. Since $I$ has dimension $0$ and annihilates $H_i(I,S)$ we have 
$h_i\leq z_i+\reg I-1$. On the other hand, the standard short exact sequences relating Koszul cycles, boundaries and homologies, give 
$z_i=b_{i-1}+1\leq \max\{ z_{i-1}+1, h_{i-1}+2\}$. Hence
$$z_i\leq 1+i+\max\{ h_j-j : j<i \}.$$
It follows that 
$$\begin{array}{rl}
h_{s+t} & \leq z_{s+t}+\reg I-1\leq z_{s}+z_{t}+\reg I-1 \\
            & \leq s+1+\max\{ h_j-j : j<s\}+ r+1+\max\{ h_j-j : j<r\}+\reg I-1.
\end{array}
$$             
\end{proof}

\begin{rem} 
In the proof of Proposition \ref{super} it is shown that  for evert $s,t$ one has 
$$\reg Z_{s+t} \leq \reg Z_s(I, Z_t) \leq \reg (Z_s\otimes Z_t)  \leq \reg Z_s +\reg Z_t$$
 provided $\dim S/I=0$.   The three  inequalities  are strict in general and this  happens already for regular sequences. 
\begin{itemize} 
\item[(1)]  For $s=t=1$, $S=\Q[a,b,c,d,e]$,  and $I=(a^2,b^2,c^2,d^2, e^2)$  one has $\reg Z_1=7$ and 
$$ \reg Z_{2}=8<\reg Z_1(I, Z_1)=11<\reg (Z_1\otimes Z_1)=13< \reg Z_1 +\reg Z_1=14.$$
\item[(2)]  If $\dim S<5$ then  $\Tor_1^S(Z_s,Z_t)=\Tor_5^S(C_{s-1},C_{t-1})=0$ where $C_i$ is the cokernel of $K(I,S)$ in position $i$. 
Hence the resolution of $Z_s\otimes Z_t$ is  the tensor product of the resolution of $Z_s$  with that of $Z_t$. It follows that 
$\reg (Z_s\otimes Z_t)=\reg Z_s +\reg Z_t$. The other two inequalities can be  strict also for $\dim S<5$. For instance, with 
$I=(a^2,b^2,c^2)\subset \Q[a,b,c]$ one has $\reg Z_2=6, \reg Z_1(I,Z_1)=9, \reg (Z_1\otimes Z_1)=2\reg Z_1=10$. 
\end{itemize} 
 \end{rem} 
     \section{Borel-fixed ideals} 

In this section, we prove Eq.\eqref{eq:4} for Borel-fixed ideals.
Throughout this section, we assume that the characteristic of $K$ is $0$.
Let $\mathrm{GL}_n(K)$ be the general linear group with coefficients in $K$.
Any $\varphi=(a_{ij}) \in \mathrm{GL}_n(K)$ induces an automorphism of $S$,
again denoted by $\varphi$,
$$\varphi \big(f(x_1,\dots,x_n) \big)=
f \left(\sum_{k=1}^n a_{k1} x_k,\dots,\sum_{k=1}^n a_{kn} x_k \right)$$
for any $f \in S$.
A monomial ideal $I \subset S$ is said to be {\em Borel-fixed}
if $\varphi(I)=I$ for any upper triangular matrix $\varphi \in \mathrm{GL}_n(K)$.
It is well-known that a monomial ideal $I \subset S$
is Borel-fixed if and only if, 
for any monomial $f x_j \in I$ and for any $i<j$, one has $fx_i \in I$.
For a monomial ideal $I$,
we write $G(I)$ for the set of minimal monomial generators of $I$.

From now on,
we fix Borel-fixed ideals $I$ and $J$
with $G(I)=\{f_1,f_2,\dots,f_m\}$
and consider the Koszul complex $K(I,S/J)= \bigwedge^\bullet F \otimes S/J$,
where $F= \bigoplus_{i=1}^m S(-\deg f_i)$.
Since Proposition \ref{linres} says that
we may assume that $I$ is generated in a single degree to prove Eq.\eqref{eq:4},
we assume $\deg f_1 = \cdots = \deg f_m$.

Let $\varphi \in \mathrm{GL}_n(K)$
be an upper triangular matrix.
Since $\varphi(J)=J$,
$\varphi$ induces an automorphism of $S/J$ defined by $\varphi(h+J)=\varphi(h)+J$.
Also, for each $f_i \in G(I)$,
since $\phi(I)=I$ we can write $\varphi(f_i)= \sum_{j=1}^m c_{ij} f_j$, where $c_{ij} \in K$, in a unique way.
We define $\varphi(e_i)= \sum_{j=1}^m c_{ij} e_j$,
and define the $K$-linear map
$$\widetilde \varphi :K_t(I,S/J) \to K_t(I,S/J)$$
by $\widetilde \varphi (e_{u_1} \wedge \cdots \wedge e_{u_t} \otimes h)
=\varphi(e_{u_1}) \wedge \cdots \wedge \varphi(e_{u_t}) \otimes \varphi(h)$.
Then, it is clear that $\widetilde \varphi \circ \phi = \phi\circ \widetilde \varphi$,
where $\phi$ is the differential of $K(I,S/J)$.
Thus we have

\begin{lem}
\label{murai1}
With the same notation as above,
$\widetilde \varphi (Z_t(I,S/J)) \subset Z_t(I,S/J)$.
 \end{lem}

Note that $\widetilde \varphi$ is actually bijective and $\widetilde \varphi (Z_t(I,S/J))=Z_t(I,S/J)$.
But we do not use this fact in the proof.

Next, we introduce a term order on $\bigwedge ^t F$.
We refer the readers to \cite{CLO} for the basics on Gr\"obner basis theory for submodules of free modules.
Let $>_{\mathrm{rev}}$ be the degree reverse lexicographic order
induced by the ordering $x_1> \cdots > x_n$.
We consider the ordering $\succ$ for the basis elements of $\bigwedge^t F$ defined by
$e_{i_1} \wedge \cdots \wedge e_{i_t} \succ  e_{j_1} \wedge \cdots \wedge e_{j_t}$,
where $i_1< \cdots < i_t$ and $j_1 < \cdots <j_t$,
if (i) $f_{i_1} \cdots f_{i_t} <_\mathrm{rev} f_{j_1} \cdots f_{j_t}$
or (ii) $f_{i_1} \cdots f_{i_t} = f_{j_1} \cdots f_{j_t}$
and $x_{i_1} \cdots x_{i_t} >_\mathrm{rev} x_{j_1} \cdots x_{j_t}$.
Then we define the term order $>$ on the free $S$-module 
$\bigwedge^t F$ defined by
$e_\ub v > e_{\ub'} v'$, where $v$ and $v'$ are monomials,
if (i) $e_\ub \succ e_{\ub'}$,
or (ii) $e_\ub = e_{\ub'}$ and $v >_\mathrm{rev} v'$.
For $g=\sum_{k=1}^l c_k e_{\ub_k} v_k$, where each $c_k \in K \setminus \{0\}$ and each $v_k$ is a monomial,
let $\ini_{>}(g) = \max_{>} \{ v_1 e_{\ub_1},\dots,v_l e_{\ub_l}\}$ be the initial term of $g$
with respect to $>$,
and for a submodule $N \subset \bigwedge^t F$,
let $\ini_>(N)=\langle \ini_>(g): g \in N \rangle$
be the initial module of $N$ with respect to $>$.

Since $J$ is a monomial ideal,
we can extend the above order $>$ to the free $S/J$-module $K_t(I,S/J)=\bigwedge^t F \otimes S/J$
in a natural way.
Thus, we call an element $v + J$ such that $v$ is a monomial of $S$ which is not in $J$
a monomial of $S/J$,
and extend the term order on $S$ to $S/J$ by identifying $v+J$ and $v$.
Since $Z_t(I,S/J)$ is a submodule of $K_t(I,S/J)$,
its initial module can be written as
\begin{eqnarray}
\label{murai4-1}
\ini_>\big(Z_t(I,S/J)\big)= \bigoplus_{\ub \subset  [m],\ \#\ub =t} e_\ub \otimes L_\ub/J,
\end{eqnarray}
where $L_\ub \subset S$ is a monomial ideal which contains $J$.

\begin{lem}
\label{murai2}
The monomial ideal $L_\ub$ in Eq.\eqref{murai4-1} is Borel-fixed.
\end{lem}

\begin{proof}
Let $v x_j \in L_\ub$ be a monomial which is not in $J$.
We prove that $v x_i \in L_\ub$ for any $i<j$ with $v x_i \not \in J$.
Let $\varphi \in \mathrm{GL}_n(K)$ be a general upper triangular matrix
and let $g \in Z_t(I,S/J)$ with $\ini_>(g)=e_\ub \otimes v x_j$.
Write
$$g= e_\ub \otimes h+ \sum_{e_{\ub'} \prec e_\ub} e_{\ub'} \otimes h_{\ub'},$$
where each $h_{\ub'} \in S/J$
and $\ini_{>_\mathrm{rev}}(h)=vx_j$.
Then, since $\varphi$ is upper triangular,
$\widetilde \varphi(g)$ can be written as
$$\widetilde \varphi(g)= e_\ub \otimes c \varphi(h)+ \sum_{e_{\ub'} \prec e_\ub} e_{\ub'} \otimes h'_{\ub'},$$
where $c \in K \setminus \{0\}$ and each $h'_{\ub'} \in S/J$.
By Lemma \ref{murai1},
$\widetilde \varphi(g) \in Z_t(I,S/J)$.
Since $Z_t(I,S/J)$ is $\mathbb{Z}^n$-graded, for each monomial $w$ which appears in $\varphi(h)$,
there is an element $g_w \in Z_t(I,S/J)$ such that $\ini_>(g_w)=e_\ub \otimes w$.
On the other hand, since $\varphi$ is general,
$v x_i$ appears in $\varphi(h)$ for any $i <j$ with $v x_i \not \in J$.
These facts prove the desired statement.
\end{proof}
 
\begin{lem}
\label{murai3}
$\ini_>(Z_t(I,S/J))$ is generated by monomials of degree $\leq t(\reg I +1) + \reg(S/J)$.
\end{lem}

\begin{proof}
We say that $e_\ub \otimes v \in K_t(I,S/J)$ divides $e_{\ub'} \otimes v' \in K_t(I,S/J)$ if $e_\ub=e_{\ub'}$ and $v$ divides $v'$.
We prove the statement by induction on $t$.
In this proof,
we assume that all the elements are homogeneous with respect to the $\mathbb Z^n$-grading.

We first consider the case $t=1$.
For a monomial $v \in S$,
we write $\max(v)$ (resp.\ $\min(v)$)
for the maximal (resp.\ minimal) integer $k$ such that $x_k$ divides $v$.
 Let
$$A=\big\{[f_i] \otimes x_k - [f_i (x_k/x_{\max(f_i)})] \otimes x_{\max(f_i)} : i=1,2,\dots,m,\ k< \max(f_i)\big\} \subset Z_1(I,S/J)$$
and
$$B=\big\{[f_i] \otimes v: i=1,2,\dots,m,\ v \in G(J:f_i)\big\} \subset Z_1(I,S/J).$$
Note that any element in $A \cup B$
has degree $\leq \reg(I)+\reg(J) =\reg(I)+1+ \reg(S/J)$.
We claim that, for any $g \in Z_1(I,S/J)$, $\ini_>(g)$ is divisible by the initial term of an element in $A \cup B$.

Let $g \in Z_1(I,S/J)$ with $\ini_>(g)=[f_i] \otimes v$.
If $\min(v) < \max(f_i)$,
it is clear that $[f_i] \otimes v$ is divisible by the initial term of an element in $A$.
Suppose $\min(v) \geq \max(f_i)$.
Then, by the definition of the ordering $\succ$,
$g$ itself is divisible by $v$.
Then $\ini_>(g/v)=[f_i] \otimes 1$,
which implies that $g/v=[f_i] \otimes 1$ since $g$ is homogeneous and $f_i \in G(I)$.
Then, since $\phi(g)=0$, we have $v f_i \in J$.
This fact says that $g$ is divisible by an elements in $B$.

Now we consider the case $t>1$.
Suppose that the statement holds for $Z_{t-1}(I,S/J)$.
Let $g \in Z_t(I,S/J)$ with $\deg (g) > t(\reg I+1) + \reg (S/J)$,
and let $\ini_>(g)=e_{u_1} \wedge \cdots \wedge e_{u_t} \otimes v$,
where $e_{u_1} \succ \cdots \succ e_{u_t}$.
We will show that there is an element $g' \in Z_t(I,S/J)$ with $\deg (g') < \deg (g)$
such that $\ini_>(g')$ divides $\ini_>(g)$.

{\em Case 1:}
Suppose $\max(v) \geq \max (f_{u_1})$.
Let $\ell= \max (v)$.
Then, by the definition of the ordering $\succ$,
$\ell= \max(v f_{u_1} \cdots f_{u_s})$ and the element $g$ is divisible by $x_\ell$.
We claim that $g/x_\ell \in Z_t(I,S/J)$.

Suppose $g/x_\ell \not \in Z_t(I,S/J)$.
Write $\phi(g/x_\ell)= \sum_{k=1}^p e_{\ub_k} \otimes c_k v_k$,
where $c_k \in K \setminus \{0\}$ and $v_k$ is a monomial in $S$
which is not in $J$.
Then $x_\ell v_k \in J$ and $\max (v_k) \leq \ell$ by the choice of $\ell$.
Also, since $\deg (g)> t(\reg I +1) + \reg(S/J)$,
$\deg v_k > \reg I + t + \reg(S/J) \geq \reg J$.
Hence there is a monomial $w_k \in J$ which strictly divides $x_\ell v_k$.
Since $J$ is Borel-fixed and $\max(v_k) \leq \ell$,
such a monomial $w_k$ can be chosen so that $w_k$ divides $v_k$,
which contradicts $v_k \not \in J$.

{\em Case 2:}
Suppose $\max(v) < \max (f_{u_1})$.
We write
$$g=a + e_1.b,$$
as in Eq.\eqref{deco}.
Then $b \in Z_{t-1}(I,S/J)$ and $\ini_>(b)=e_{u_2} \wedge \cdots \wedge e_{u_s} \otimes v$.
By the induction hypothesis,
there is an $h \in Z_{t-1}(I,S/J)$ with $\deg h \leq (t-1) (\reg I+1) + \reg(S/J)$
such that $\ini_>(h)$ divides $\ini_>(b)$.
Let $\ini_>(h)= e_{u_2} \wedge \cdots \wedge e_{u_s}  \otimes \delta$.
Take an element $x_i$ which divides $v / \delta$.
 
Observe that $[f_{u_1}]\otimes x_i - [f_{u_1}(x_i/{x_{\max(f_{u_1})}})] \otimes x_{\max(f_{u_1})} \in Z_1(I,S/J)$
since $i \leq \max(v) < \max(f_{u_1})$.
Then the element
$$g'=([f_{u_1}]\otimes x_i - [f_{u_1}(x_i/{x_{\max(f_{u_1})}})] \otimes x_{\max(f_{u_1})}).h\in Z_1(I,S/J)Z_{t-1}(I,S/J) \subset Z_t(I,S/J)$$
satisfies the desired conditions.
Indeed, $\ini_>(g')=e_{u_1} \wedge \cdots \wedge e_{u_t} \otimes \delta x_i$ divides $\ini_>(g)$
and
$$\deg (g') \leq \reg(I)+1+ (t-1)(\reg I +1) + \reg(S/J) = t(\reg I +1) + \reg(S/J)< \deg(g),$$
as desired.
\end{proof}

\begin{thm}
\label{thm2}
Let $I$ and $J$ be Borel-fixed ideals.
Then 
$$\reg Z_t(I,S/J) \leq t(\reg I +1) + \reg(S/J).$$
\end{thm}

\begin{proof}
By Proposition \ref{linres}, we may assume that $I$
is generated in a single degree.
Consider the decomposition 
Eq.\eqref{murai4-1} before Lemma \ref{murai2}.
Then we have
\begin{align}
\label{murai4-2}
\reg(Z_t(I,S/J)) \leq \reg(\ini_> (Z_t(I,S/J))) = \max\{ \reg e_\ub \otimes L_\ub/J: \ub \subset [m],\ \#\ub=t\}.
\end{align}
On the other hand, by Lemmas \ref{murai3},
each $e_\ub \otimes L_\ub/J$ is generated by elements of degree $\leq t(\reg I+1) +\reg(S/J)$.
Thus $L_\ub$ is generated by monomials of degree $\leq t+ \reg S/J$.
Since $L_\ub$ is Borel-fixed by Lemma \ref{murai2},
the result of Eliahou and Kervaire \cite{EK}
shows that $\reg L_\ub \leq t +\reg(S/J)$.
Also the short exact sequence
$$0 \longrightarrow J \longrightarrow L_\ub \longrightarrow L_\ub/J \longrightarrow 0$$
shows $\reg L_\ub/J \leq \max\{\reg J-1,\reg L_\ub\} \leq t+ \reg(S/J)$.
Then the desired statement follows from Eq.\eqref{murai4-2}.
\end{proof}

From the above theorem,
we get the next bound for the regularity of Koszul homology.

\begin{cor}
\label{cor3}
Let $I$ and $J$ be Borel-fixed ideals.
Then 
$$\reg H_t(I,S/J) \leq (t+1)(\reg I +1) + \reg(S/J) -2.$$
\end{cor}

\begin{proof}
Let $b_i=\reg B_i(I,S/J)$, $z_i=\reg Z_i(I,S/J)$ and $h_i= \reg H_i(I,S/J)$.
Then,  the standard short exact sequences relating Koszul cycles, boundaries and homologies
show that $b_{i}= z_{i+1}-1$ and $h_i \leq \max\{b_{i}-1,z_i\}$
for all $i$.
Hence, by Theorem \ref{thm2}, we have
$$\reg H_t(I,S/J)=h_t \leq \max\{z_{t+1}-2,z_t\} \leq(t+1)(\reg I +1) + \reg(S/J) -2,$$
as desired.
  \end{proof}

\end{document}